\documentclass[]{siamart190516}
\usepackage{graphicx}
\usepackage{amsmath, amssymb}
\usepackage{hyperref}
\usepackage{thmtools, mathtools}
\usepackage{algorithm,algorithmic}
\usepackage{placeins}
\usepackage[]{caption}
\usepackage{subcaption}
\usepackage{chngcntr}
\usepackage[utf8]{inputenc}
\usepackage{dirtytalk} %makes the "" look as expected
\counterwithin{table}{section}
\counterwithin{figure}{section}

\DeclarePairedDelimiter\abs{\lvert}{\rvert}
\makeatletter
\let\oldabs\abs
\def\abs{\@ifstar{\oldabs}{\oldabs*}}

\makeatletter
\let\c@table\c@figure % for (1)
\let\ftype@table\ftype@figure % for (2)
\makeatother

\newtheorem{define}[theorem]{Definition}

\newcommand{\R}{\mathbb{R}}

\DeclareMathOperator{\E}{\mathbb{E}}   % expected value 
\DeclareMathOperator{\fl}{\mathrm{fl}}  % floating point number

\let\outputs\algorithmicensure{ }

\title{ 
Probabilistic Error Analysis For Sequential Summation of Real Floating Point Numbers\thanks{Funding: This work was supported in part by National Science Foundation grant DMS-1760374.}
}
 
\author{Johnathan Rhyne\thanks{Department of Mathematics, North Carolina State University, Raleigh, NC 27695-8205, USA, (jprhyne@ncsu.edu)}} % insert in alphabetical order

\begin{document}
    \maketitle
    \begin{center}
    	Advisor: Ilse C.F. Ipsen\footnote{Department of Mathematics, North Carolina State University, Raleigh, NC 27695-8205, USA, (ipsen@ncsu.edu)}
    \end{center}
    \begin{abstract} We propose probabilistic models to bound the forward error in the numerically computed sum of a vector with $n$ real elements. To do so, we generate our own deterministic bound for ease of comparison, and then create a model of our errors to generate probabilistic bounds of a comparable structure that can typically be computed alongside the actual computation of the sum. \\
    \indent The errors are represented as bounded, independent, zero-mean random variables. We have found that our accuracy is increased when we use vectors that do not sum to zero. This accuracy reaches to be within 1 order of magnitude when all elements are of the same sign. We also show that our bounds are informative for most cases of IEEE half precision and all cases of single precision numbers for a for a vector of dimension, $n \leq 10^7$. \\
    \indent Our numerical experiments confirm that the probabilistic bounds are tighter by 2 to 3 orders of magnitude than their deterministic counterparts for dimensions of at least 80 with extremely small failure probabilities. The experiments also confirm that our bounds are much more accurate for vectors consisting of elements of the same sign.
    \end{abstract}
    

    \begin{AM}
        65F99, 65G50, 60G42, 60G50
    \end{AM}
	\section{Introduction}
    \paragraph{Background}
	We consider the summation of $n$ real numbers in IEEE floating point arithmetic. First, we generate a deterministic bound for the relative error, introduce two probabilistic inequalities that we need to create our probabilistic bounds, and we confirm that our probabilistic bounds are tighter than their deterministic counterparts typically by 2 orders of magnitude. Finally, we note both the best cases of our bounds and under what conditions they fail along with future directions to possibly make better bounds.
    \paragraph{Other Work} There have been probabilistic approaches for roundoff analysis applied to inner products (Ipsen and Zhou \cite{Ipsen2020}), matrix inversion (von Neumann \& Goldstine \cite{Neumann1947} and Tienari \cite{Tienari1970}), and LU decomposition and linear system solutions (Babuška \& S\"oderlind \cite{Babuska2018},  Higham \& Mary \cite{Higham2019}), and Barlow \& Bareiss \cite{Bareiss1980, Barlow1985, Barlow1985a}. Higham \& Mary also approach probabilistic analysis for summations in \cite[Section 2]{Higham2020}, however their model takes into account the structure of the data being summed, so a meaningful comparison is difficult. Although roundoff errors do not behave as random variables, \cite[Page 2]{Hull1966} and \cite[Page 17]{Kahan1996}, our bounds give much more realistic results than deterministic ones, and we confirm this for $n \leq 10^7$ with our numerical experiments.
    \paragraph{Contributions}
	I was inspired by the work in \cite{Ipsen2020} on roundoff errors in the computation of inner products. There, roundoff errors are represented as random variables and the probabilistic bounds are derived from concentration inequalities. Here, we construct two models of our roundoff errors and also apply probabilistic concentration inequalities. Our probabilistic bounds also are expressed in terms of a chosen failure probability, so we can more easily see how our failure tolerance affects the accuracy of the bounds.
	
	\section{Floating Point Arithmetic} \label{sec:FPA}
	Before deriving the bounds, we give an example of floating point arithmetic. Since a computer has a finite amount of storage, it is necessary to approximate real numbers like $\frac{1}{5}$ and $\frac{1}{3}$.
	The IEEE society defines conventions for storing real numbers as sums of powers of 2 \cite{IEEE}	
	 Figure \ref{fig:FPNum} shows the example of $\frac{5}{32}$. Therefore, numbers like $\frac{1}{3}$
	 cannot be represented exactly in floating point arithmetic, and rounding error occur when storing the numbers that are not exact sums of powers of two.

	\begin{figure}[H]
		\centering
		\includegraphics[width=.75\textwidth]{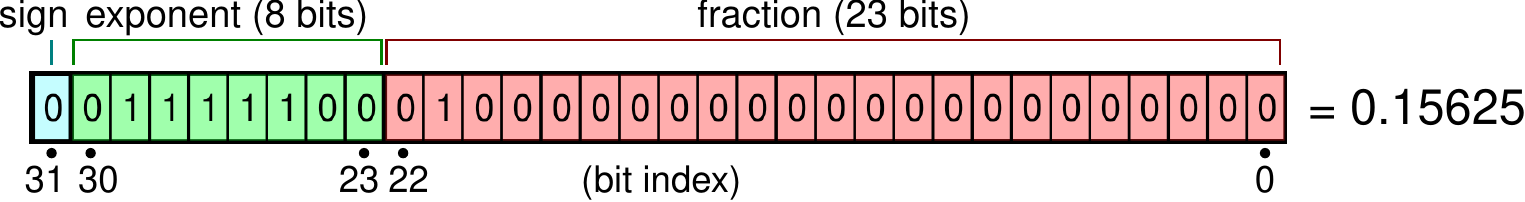}
		\caption{Artistic rendition of how $\frac{5}{32}$ is stored as a 32-bit floating point number defined \cite{IEEE}. Image courtesy of \cite{FloatExample}.} \label{fig:FPNum}
	\end{figure}
	
The difference between 1 and the next largest floating number is called \textit{machine epsilon}, $\varepsilon$, 
and the \textit{unit roundoff} is half of machine epsilon, $u=\varepsilon/2$ \cite[Section 2.1]{Higham2002}.
% avoid unnecessary details, like conventional rounding, etc 

Table \ref{tab:UnitRoundoff} presents the unit roundoff for floating
point numbers in single and half precision.
	\begin{table}
		\begin{center}
			\caption{Unit roundoff for IEEE \cite{IEEE} single and half precision floating point numbers.} \label{tab:UnitRoundoff}
			\begin{tabular}{|l|r|}
			\hline
			\textbf{Half Precision} & \textbf{Single Precision} \\
			\hline
			$u = 2^{-11}$ & $u = 2^{-24}$ \\
			\hline
			\end{tabular}
		\end{center}
	\end{table}

In this paper, we consider roundoff errors that occur in addition, and use the standard model
\cite[(2.4)]{Higham2002}
	\begin{equation} \label{eq:FPModel}
		\fl(a + b) = (a + b)(1 + \delta) = a + b + \delta(a + b)  \qquad \text{where}\quad \abs{\delta} \leq u.
	\end{equation}
Intuitively, this means floating point addition is exact addition with some small error that is scaled by the numbers 
that are being added.

	\section{Deterministic Bounds} \label{sec:DetBounds} 
 	We derive one deterministic bound based on the actual errors incurred during floating point addition.
 
 Given $n$ real numbers $x_k$, $1\leq k\leq n$, 
 we compute their sum $\sum_{k=1}^{n}{x_k}$ by adding one element after the other.
 Table \ref{tab:FPRep} shows the order of summation in exact arithmetic, and
 the roundoff errors in floating point arithmetic.
 We denote by $\fl\left(\sum_{k=1}^{n}{x_k}\right)$ the sum computed in floating point arithmetic.
 
    \paragraph{Assumptions} We assume that the $x_k$ are floating point numbers 
    so they can be stored exactly and do not incur representation errors. The only roundoff errors occur during the  summation,
    as shown in (\ref{eq:FPModel}).  The roundoff errors in the individual additions are $\delta_k$, where 
    $\abs{\delta_k} \leq u$, $2\leq k\leq n$, as described in Section~\ref{sec:FPA}.
    We also assume $\sum_{k=1}^{n}{x_k} \neq 0$ so that the relative error is defined. If the sum is $0$, then 
 we can remove the denominator to obtain absolute bounds.

    %\FloatBarrier
    \begin{table}[H]
    \begin{center}
    \caption{Exact and computed sums}
    \label{tab:FPRep} %Add Valid range on this table
    \begin{tabular}{|l|r|r|} 
      \hline
      \textbf{Exact computation} & \textbf{Floating point arithmetic}&\textbf{Index range}\\
      \hline
      $z_1 = x_1$ & $\hat{z}_1 = x_1$ &\\
      $z_2 = x_1 + x_2$ & $\hat{z}_2 = (x_1 + x_2)(1 + \delta_2)$&\\
      $z_k = z_{k-1} + x_k$ & $\hat{z}_k =  (\hat{z}_{k-1} + x_k)(1 + \delta_k)$& $2 \leq k \leq n$\\
      \hline
      $z_n = \sum_{k=1}^{n}{x_k}$ & $\hat{z}_n = \fl\left(\sum_{k=1}^{n}{x_k}\right)$&\\
      \hline
    \end{tabular}
    \end{center}
    \end{table}
    
  In Table~\ref{tab:FPRep}, $z_k$ represents the exact $k^{th}$ partial sum, while $\hat{z}_k$ represents the 
    $k^{th}$ partial sum computed in floating point arithmetic. 

	\begin{table}[H]
    \begin{center}
    \caption{Variables in the first representation and range of indices}
    \label{tab:variableConstruction}
    \begin{tabular}{|l|r|} 
      \hline
      \textbf{Construction} & \textbf{Index range}\\
      \hline
      $c_1 = 0$ &\\
      $c_k = u\sum_{\ell=1}^k{\abs{x_\ell}} = u\|x_{(k)}\|_1$ & $2 \leq k \leq n$ \\
      \hline
      $Z_k = \delta_k\sum_{\ell=1}^kx_\ell$ & $2 \leq k \leq n$\\
      \hline
      $Z = \sum\limits_{j=2}^{n}{Z_j}$ = $\fl\left(\sum_{k=1}^{n}{x_k}\right) - \sum_{k=1}^{n}{x_k}$& \\
      \hline
    \end{tabular}
    \end{center}
    \end{table}

In Table~\ref{tab:variableConstruction}, $Z$ represents the absolute error in the floating point sum,
$Z_k$ represents all the terms affected by the $k^{th}$ error,  and $c_k$ is an
 upper bound for $|Z_k|$. Note that all summands in $c_k$ are multiples of $u$, that is, 
 $c_k =|x_k|(u+\cdots)$.

    \begin{theorem} \label{thm:cDetBound}  For $1 \leq k \leq n$, let $x_k \in \R$ where $\sum_{k=1}^{n}{x_k} \neq 0$, $c_k$ in Table \ref{tab:variableConstruction}, $\hat{z}_n$ and $z_n$ in Table \ref{tab:FPRep}. 
    Then
    \begin{equation} \label{eq:cDetBound}
            \abs{\dfrac{\hat{z}_n - z_n}{z_n}} \leq \dfrac{{\sum_{k=2}^{n}{c_k}}}{\abs{z_n}}.
        \end{equation}
    \end{theorem}

	\begin{proof}
		We first show that the $Z_k$ values are bounded above by $c_k$ for $2 \leq k \leq n$. These bounds follow directly from our assumption on our roundoff errors, $\abs{\delta} \leq u$.
		\[
                    \abs{Z_k} = \abs{\delta_k\sum_{\ell=1}^kx_\ell} \leq u\sum_{\ell=1}^k\abs{x_k} = c_k.
		\]
		Next, we apply these bounds to the final sum, $Z$,
		
		\[
        	\abs{\hat{z}_n - z_n} = {\abs{Z}} \leq {\sum_{k=2}^n{c_k}}.
        \]
        Finally, we divide both sides by $\abs{z_n}$ to get our desired bound of
        \[
        	\abs{\dfrac{\hat{z}_n - z_n}{z_n}} \leq \dfrac{{\sum_{k=1}^{n}{c_k}}}{\abs{z_n}}.
        \]
        
	\end{proof}
	
  	By the construction of each $c_k$ in Table \ref{tab:variableConstruction}, $u$ is multiplied by the one norm of our data. 
        This is the connection between this bound and the precision of floating point being used. So, when $u$ is relatively large when compared to 
        $n$, as is noted for the half precision case in Table \ref{tab:UnitRoundoff}, we would want a bound that is much tighter since we can't 
        rely on the small nature of $u$. To achieve this, we must relinquish bounding the actual errors themselves, and instead bound a 
        probabilistic model of these errors.
    
	\section{Background Information} \label{sec:BI}
    Before we move to our probabilistic bounds, we explicitly state Azuma's inequality and the Azuma-Hoeffding Martingale inequality along with any relaxed assumptions for our case. We also use $\E[Z]$ to denote the expected value of the random variable $Z$.
    \begin{lemma}[Theorem 5.3 in \cite{Chung2006}] \label{thm:Azuma} 
    Let $A = A_1 + \cdots + A_n$ be a sum of independent real-valued random variables,  $0 \leq a_k$ for $1 \leq k \leq n$, $0 < \delta < 1$, and
        \[
            \abs{A_k - \E[A_k]} \leq a_k \quad \quad 1 \leq k \leq n.
        \]
        Then with probability at least $1 - \delta$
        \[
            \abs{A - \E[A]} \leq \sqrt{2\ln{\frac{2}{\delta}}}\sqrt{\sum\limits_{k=1}^{n}{a_k^2}}.
        \]
        \end{lemma}  
        \begin{proof} We have two cases to consider. The first being if $c_k = 0$ for $1 \leq k \leq n$.
        If this is the case, then we set the right hand side to be $0$ because this means, from the linearity of the expectation,
        \[
        	\sum_{k=1}^{n}{A_k} = \E\left[\sum_{k=1}^{n}{A_k} \right] \implies A = \E[A] \implies \abs{A - \E[A]} = 0.
		\]
        Otherwise, in \cite[Theorem 5.3]{Chung2006} set
        \[
            \delta = \text{Pr}[\abs{A - \E[A]} \geq t] \leq 2\text{exp}\left( -\dfrac{t^2}{2\sum_{k=1}^{n}{a_k^2}} \right).
        \]
        Then, we solve for $t$ in terms of $\delta$ to get 
		\[
			t \leq \sqrt{2\ln\frac{2}{\delta}}\sqrt{\sum\limits_{k=1}^{n}{a_k^2}}.
		\]
                Since $\delta$ is the probability that $\abs{A - \E[A]} \geq t$, then $\abs{A - \E[A]} \leq t$ holds with probability at least $1 - \delta$.
        \end{proof}

		This inequality tells us, for any fixed failure probability $\delta$, how much a sum of random variables differs from its expectation given each summand has bounded difference from its own expectation. 
      
	Before we can move to our next inequality, we first define a Martingale.
	\begin{define}[Theorem 12.1 in \cite{Mitzenmacher2006}] \label{def:Martingale} 
        A collection of random variables, $B_1,B_2,\dots, B_n$ is called Martingale if the following are satisfied
        \begin{enumerate}
            \item $\E[\abs{B_n}]$ is finite.
            \item $\E[B_{k}|B_1,\dots,B_{k-1}] = B_{k-1}$
        \end{enumerate}
        This is also referred to as being a Martingale with respect to itself.
    \end{define}
	
    \begin{lemma}[Theorem 12.4 in \cite{Mitzenmacher2006}] \label{thm:Hoeffding}
    Let $B_1 = 0$, $B_2, \, \dots, \, B_{n} $ be a Martingale with respect to itself, as in Definition \ref{def:Martingale}. Also let $0 \leq b_k$ for $1 \leq k \leq n-1$, and $0 < \delta < 1$
    If
    \[
        \abs{B_k - B_{k-1}} \leq b_{k-1} \qquad \text{for } 2 \leq k \leq n,
    \]
    then with probability at least $1 - \delta$
    \[
        \abs{B_{n} - B_1} \leq \sqrt{2\ln{\frac{2}{\delta}}}\sqrt{\sum\limits_{k=1}^{n-1}{b_k^2}}.
    \]
    \end{lemma}
    \begin{proof}
    	Much like in Lemma \ref{thm:Azuma}, we first address the case of $b_k = 0$ for $2 \leq k \leq n$. If so, we set the right hand side to be $0$ as this would mean that each $B_k$ is itself $0$. Otherwise, we do the following. 
    	In \cite[Theorem 12.4]{Mitzenmacher2006}, set
    	\[
    		\delta = \text{Pr}[\abs{B_n - B_1} \geq t] \leq 2\text{exp}\left( -\dfrac{t^2}{2\sum_{k=1}^{n-1}{b_k^2}} \right).
    	\]
    	 Since $\delta$ is the probability that $\abs{B_n - B_1} \geq t$, then  $\abs{B_n - B_1} \leq t$ holds with probability $1 - \delta$.
    \end{proof}

     The modification that we make for the preceding Lemma is that our indexing of the $b_k$ values range from $1 \leq k \leq n-1$ instead of 
     $1 \leq k \leq n$. This does not change the result since \cite{Mitzenmacher2006} suggests defining the first element to be zero, which means 
     we only have $n-1$ non-zero values to be bounded. The difference between Lemmas \ref{thm:Azuma} and \ref{thm:Hoeffding} is that the former 
     tells us the probability that a sum of random variables differs from its expected value, while the latter tells us how likely the $n^{th}$ 
     term differs from the first term with a chosen failure probability $\delta$.
    
    \section{Probabilistic Bounds}\label{sec:ProbBounds} Moving forward, instead of bounding the actual relative errors that are incurred, we instead 
    model our roundoff errors as zero mean random variables. To reduce extraneous variables, we use the same notation as in Tables \ref{tab:FPRep} 
    and \ref{tab:variableConstruction}. This change comes with the following assumptions, for $2 \leq k \leq n$, $\E[\delta_k] = 0$ and each 
    $\delta_k$ are mean independent. Now, we show that this new model, with the same representation as in Tables \ref{tab:FPRep} and 
    \ref{tab:variableConstruction}, gives us much tighter bounds than in Section \ref{sec:DetBounds}.
    
	Before we construct our bounds, we show that $Z$, in Table \ref{tab:variableConstruction}, satisfies the requirements to be used in Lemma \ref{thm:Azuma}.
	We already have the bounded property from the proof of Theorem \ref{thm:cDetBound}, so we need only show the zero mean property.
	
	\begin{lemma}\label{thm:ZProps}
		Assuming $Z$ and $Z_k$ for $2 \leq k \leq n$ in Table \ref{tab:variableConstruction}. Then
		\[
			\E[Z] = 0 \text{ and } \E[Z_k] = 0. \qquad 2 \leq k \leq n
		\]
		\end{lemma}
		\begin{proof}
			From our construction of $Z_k$ for $2 \leq k \leq n$, $x_k$ being constants, and from the independence of $\delta_k$, we have
			\[
                            \E[Z_k] = \E\left[\delta_k\sum_{\ell=1}^kx_\ell\right] = \E\left[\delta_k\right]\sum_{\ell=1}^kx_\ell = 0
                        \]
			From the linearity of the expectation
			\[
				\E[Z] = \E\left[\sum\limits_{j=2}^{n}{Z_j}\right] = 0.
			\]
		\end{proof}

     Now, we construct the probabilistic counterpart of \eqref{eq:cDetBound}.
    
    \begin{theorem}\label{thm:cProbBound}
        Let $Z$ in Table \ref{tab:variableConstruction}, $x_k \in \R$ for $1 \leq k \leq n$ ,  $\sum_{k=1}^{n}{x_k} \neq 0$, and $0 < \delta < 1$. Then with probability at least $1 - \delta$
        \begin{equation}\label{eq:cProbBound}
            \abs{\dfrac{\hat{z}_n - z_n}{z_n}} \leq \dfrac{\sqrt{2\ln{\frac{2}{\delta}}}\sqrt{\sum\limits_{k=1}^{n}{c_k^2}}}{\abs{z_n}}.
        \end{equation}
    \end{theorem}
        \begin{proof}
		By applying Lemma \ref{thm:Azuma},
    \[
		\abs{\hat{z}_n - z_n} = \abs{Z} \leq \sqrt{2\ln{\frac{2}{\delta}}}\sqrt{\sum\limits_{k=1}^{n}{c_k^2}}.
    \]
    Then, we divide both sides by $\abs{z_n}$ to get our desired bound
    \[
    		\abs{\dfrac{\hat{z}_n - z_n}{z_n}} \leq \dfrac{\sqrt{2\ln{\frac{2}{\delta}}}\sqrt{\sum\limits_{k=1}^{n}{c_k^2}}}{\abs{z_n}},
    \]
    which holds with probability at least $1 - \delta$.
    \end{proof}
    %Better way writing this transition.
    Before we construct our final probabilistic bound, we first consider the following random variables.
    \begin{table}[H]
    \begin{center}
    \caption{\\ Errors in the Partial Sums and their Bounds}
    \label{tab:MDef}
    \begin{tabular}{|l|r|} 
      \hline
      \textbf{Construction} & \textbf{Valid Range}\\
      \hline
      $M_k = \hat{z}_k - z_k$ & $1 \leq k \leq n$ \\
      $M_n = \hat{z}_n - z_n$ &\\
      \hline
      $m_k = \abs{x_1}(1 + u)^{k-1} + \sum_{j=2}^{k+1}{\abs{x_j}(1+u)^{k-j+1}}$ & $1 \leq k \leq n-1$ \\
      \hline
    \end{tabular}
    \end{center}
    \end{table}
    
    Intuitively, each $M_k$ value represents the difference between the $k^{th}$ floating point and exact partial sums, which gives the error in our floating point and exact sums. The $m_k$ represent an upper bound for the errors in the $(k + 1)^{th}$ partial sum before we account for the $(k + 1)^{th}$ roundoff error, and $M_1 = \hat{z}_1 - z_1 = x_1 - x_1 = 0.$
    
	Using Table \ref{tab:MDef}, we bound a telescoping sum of each $M_k$, so we first provide a recursive characterization of our $M_k$, and then bound the absolute difference of $M_k$ and $M_{k-1}$ for $2 \leq k \leq n$.
    \begin{lemma} \label{thm:Recursion} Let $M_k$ in Table \ref{tab:MDef}, $\hat{z}_k$ in Table \ref{tab:FPRep}, and $x_k \in \R$. Then
        $M_k$ satisfies the recursion
        \[
            M_k = M_{k-1} + \delta_k(\hat{z}_{k-1} + x_k) \qquad 2 \leq k \leq n.
        \]
    \end{lemma} 
    \begin{proof}
        We substitute in the recursive definition of $\hat{z}_k$ from Table \ref{tab:FPRep} and get
        \begin{align*}
            M_k &= (\hat{z}_{k-1} + x_k)(1 + \delta_k) - (z_{k-1} + x_k) \\
            &= M_{k-1} + \delta_k(\hat{z}_{k-1} + x_k) \qquad 2 \leq k \leq n.
        \end{align*}
    \end{proof}    

	In order to bound the absolute difference between $M_k$ and $M_{k-1}$ for $2 \leq k \leq n$, we first need to show the bounded nature of each $\hat{z}_k$
   
   \begin{lemma} \label{thm:zHatBound}
The floating point sums $\hat{z}_k$ in Table \ref{tab:FPRep} are bounded by\footnote{The sum is zero if the lower limit
exceeds the upper limit.}
    	\[
    		\abs{\hat{z}_k} \leq \abs{x_1}(1 + u)^{k-1} + \sum\limits_{j=2}^{k}\Big({\abs{x_j}(1 + u)^{k - j + 1}}\Big), 
		\qquad 1 \leq k \leq n.
    	\]
    \end{lemma}
    	\begin{proof}
    		From Table \ref{tab:FPRep} follows the explicit representation for the computed partial sums,
    		\[
    			\hat{z}_k = x_1\prod\limits_{l=2}^{k}{(1 + \delta_l)} + \sum\limits_{j=2}^{k}\left({x_j\prod\limits_{l=j}^{k}{(1 + \delta_l)}}\right), \qquad 1 \leq k \leq n.
    		\]
 Applying the upper bounds $\abs{\delta_k} \leq u$ to the above expressions gives
    		\[
    			\abs{\hat{z}_k} \leq \abs{x_1}(1 + u)^{k-1} + \sum\limits_{j=2}^{k}\Big({\abs{x_j}(1 + u)^{k - j + 1}}\Big) \qquad 2 \leq k \leq n.
    		\] 
    	\end{proof}
   
   Next, we use Lemma \ref{thm:zHatBound}, to get an upper bound for the absolute difference between $M_k$ and $M_{k-1}$ for $2 \leq k \leq n$
    
    \begin{lemma}\label{thm:MkBound} Let $M_k$ and $m_k$ be defined by Table \ref{tab:MDef}. Then
        \[
            \abs{M_k - M_{k-1}} \leq um_{k-1} \qquad 2 \leq k \leq n.
        \]
    \end{lemma}
        \begin{proof}
        We recall that  for $2 \leq k \leq n$, $\abs{\delta_k} \leq u$.
            \begin{align*}
                \abs{M_k - M_{k-1}} &= \abs{\delta_k(\hat{z}_{k-1} + x_k)} \\
                &\leq u\Big(\abs{x_1}(1 + u)^{k-2} + \sum\limits_{j=2}^{k-1}{\abs{x_j}(1 + u)^{k - j}} + \abs{x_k}\Big) \\
                &= u\Big(\abs{x_1}(1 + u)^{k-2} + \sum\limits_{j=2}^{k}{\abs{x_j}(1 + u)^{k - j}}\Big) = um_{k-1} \qquad 2 \leq k \leq n.
            \end{align*}
            Therefore, we have our desired result of 
            \[
                \abs{M_k - M_{k-1}} \leq um_{k-1},\qquad 2 \leq k \leq n.
            \]
        \end{proof}
    
    We now show that the collection, $M$, of $M_k$ for $1 \leq k \leq n$ is a Martingale by Definition \ref{def:Martingale}. The requirement of $\E[M_n]$ being finite is satisfied since for $1 \leq k \leq n$, $M_k$ is a difference between two bounded numbers. We know both $z_k$ and $\hat{z}_k$ are bounded since they are sums of $k$ finite floating point numbers. Finally, we can say that each $\delta_k$ is independent of $M_1,\dots,M_{k-1}$ for $2 \leq k \leq n$ because each $M_k$ is written in terms of $\hat{z}_k$ and $z_k$, where the former  is in terms of constants and $\delta_l$ for $1 \leq l \leq k$, and the latter is a constant.
    Our next task is to prove $M$ satisfies the second requirement of Definition \ref{def:Martingale}.
    \begin{lemma} \label{thm:Martingale}
		For $2 \leq k \leq n$, let $M_k$ be defined by Table \ref{tab:MDef}, $\delta_k$ be zero mean, independent random variables, and $M_1 = 0$.
    \end{lemma}
        \begin{proof}
        We recall the linearity of the expectation, and Lemma \ref{thm:Recursion}.
        First we show an explicit case for $k=2$. 
        \begin{align*}
            \E[M_2|M_1] &= \E[M_1 + \delta_2(x_1 + x_2)|M_1] \\
            &= \E[M_1|M_1] + \E[\delta_2|M_1](x_1 + x_2) \\
            &= M_1 +  \E[\delta_2](x_1 + x_2) = M_1.
        \end{align*}
        Now, let $3 \leq k \leq n$
        \begin{align*}
            \E[M_k|M_1,\dots,M_{k-1}] &= \E[M_{k-1} + \delta_k(\hat{z}_{k-1} + x_k)|M_1,\dots,M_{k-1}] \\
            &= \E[M_{k-1}|M_1,\dots,M_{k-1}] + \E[\delta_k|M_1,\dots,M_{k-1}](\hat{z}_{k-1} + x_k) \\
            &= M_{k-1} \qquad 3 \leq k \leq n.
        \end{align*}
        
        \end{proof}

    Now, we know that $M$ forms a Martingale. So, we show our second probabilistic bound by applying Lemma \ref{thm:Hoeffding}.
    
    \begin{theorem}\label{thm:mProbBound}
        Let $x_k \in \R$ for $1 \leq k \leq n$ where $\sum_{k=1}^{n}{x_k} \neq 0$, $m_k$ in \ref{tab:variableConstruction}, and $0 < \delta < 1$.
        Then with probability at least $1 - \delta$
        \begin{equation}\label{eq:mProbBound}
            \dfrac{\abs{\hat{z}_n - z_n}}{\abs{z_n}} \leq \dfrac{u\sqrt{2\ln{\frac{2}{\delta}}}\sqrt{\sum_{k=1}^{n-1}{m_k^2}}}{\abs{z_n}}.
        \end{equation}
    \end{theorem}
        \begin{proof}
            We first recognize that by construction of $M$, we know that $M_1 = x_1 - x_1 = 0$. Then we apply Lemma \ref{thm:Hoeffding}.
            \begin{align*}
                \abs{\hat{z}_n - z_n} &= \abs{M_n} = \abs{M_n - M_1} \\
                &\leq u\sqrt{2\ln{\frac{2}{\delta}}}\sqrt{\sum\limits_{k=1}^{n-1}{m_k^2}}.
            \end{align*}
            Then, we divide both sides by $\abs{z_n}$ to get our desired bound,
            \[
            		\dfrac{\abs{\hat{z}_n - z_n}}{\abs{z_n}} \leq \dfrac{u\sqrt{2\ln{(\frac{2}{\delta})}}\sqrt{\sum_{k=1}^{n-1}{m_k^2}}}{\abs{z_n}},
            \] 
            which holds with probability at least $1 - \delta$.
        \end{proof}
	\paragraph{Important note} Due to the definition of $m_k$ in Table \ref{tab:MDef}, the computation of \eqref{eq:cProbBound}, we are unable to calculate this bound for large dimensions, above approximately $n = 10^6$, in an informative time frame. A potential workaround is discussed in Section \ref{sec:Conc}.
    \section{Numerical Experiments} \label{sec:Numerical Experiments}
    \subsection{Hardware}\label{sec:Hardware} 
    The final results were gathered with \verb|matlab2018a| under the NC State HPC\footnote{More information can be found at 
    \url{https://wp.math.ncsu.edu/it/high-powered-computing-cluster/}}
    \subsection{Design} Our experiments use vectors of normally distributed random variables with mean zero and variance of one, and uniformly distributed random variables from zero to one. To ensure that RAM capacity is not an issue, we generate each of the $n$ elements during the function that computes our bounds. To ensure repeatability, we used the Random package to seed before each bound calculation call. The seed used is 123. Specifics for the programs used are in Section \ref{sec:Algorithms}. In our experiments, unit roundoff, denoted $u$, is as stated in Table \ref{tab:UnitRoundoff}. These algorithms are designed to scale to an arbitrary precision, however we only use the case of the inputs being 32 bit floating point. If you want to scale to a floating point with more precision digits than the IEEE 64 bit floating point, assign the \say{true} precision in the algorithms to 256 bit precision. The choice to use 256 bits for precision is simply because 256 bit is the default for Julia's BigFloat data type.
    \subsection{Algorithms} \label{sec:Algorithms}
    In this section, we will first list the two algorithms for generating our bounds. Next, we will describe how we generate a graph, display the graph, and give an interpretation of the results. Finally, we will make final interpretations.
    \paragraph{Naming of Bounds}
    Within the bounds, we refer to the below equation numbers for the bounds for ease of reference.
    \begin{align}
		\abs{\dfrac{\hat{z}_n - z_n}{z_n}} &\leq \sqrt{n}\dfrac{{\sum_{k=1}^{n}{c_k}}}{\abs{z_n}},  \label{eq:cDetBoundGraphs}\\
		\abs{\dfrac{\hat{z}_n - z_n}{z_n}} &\leq \sqrt{2\ln{\frac{2}{\delta}}}\dfrac{\sqrt{\sum_{k=1}^{n}{c_k^2}}}{\abs{z_n}}, \label{eq:cProbBoundGraphs}\\
		\abs{\dfrac{\hat{z}_n - z_n}{z_n}} &\leq u\sqrt{2\ln{\frac{2}{\delta}}}\dfrac{\sqrt{\sum_{k=1}^{n-1}{m_k^2}}}{\abs{z_n}}. \label{eq:mProbBoundGraphs}
    \end{align}
        
    \paragraph{Common Inputs} By design, all of our algorithms have the same inputs, which we will list here to save space.
    \begin{enumerate}
        \item $n$: the dimension of vector we are simulating
        \item $\delta$: the probability of failure as in Lemma \ref{thm:Azuma}. We use $10^{-16}$ in our graphs
        \item T: the type of Floating Point (FP) number we want. The following graphs use 32-bit floating point, however, the algorithm works for any FP type as long as one can determine machine epsilon.
        \item rf: The function that is used to generate the variables. We use Julia's Base.randn function for this generation unless otherwise specified.
    \end{enumerate}
    \paragraph{Common Output} By design, we also return the relative error in all of our algorithms.
    \begin{itemize}
    	\item Relative Error: $\abs{\dfrac{\hat{z}_n - z_n}{z_n}}$
    \end{itemize}
         
    \begin{algorithm} %1
        \caption{Computation of \eqref{eq:cDetBoundGraphs}, \eqref{eq:cProbBoundGraphs}, and the relative error for some $n\in \mathbb{N}$}
        \outputs{} Equations \ref{eq:cDetBoundGraphs}, \ref{eq:cProbBoundGraphs}, and relErr 
        \begin{algorithmic}[1]
            \IF{T is 64 bit Floating Point} 
            \STATE TTrue $\leftarrow$ 256 bit Floating Point
            \ELSE
            \STATE TTrue $\leftarrow$ 64 bit Floating Point
            \ENDIF
            \STATE xSum $\leftarrow$ 0 of type T
            \STATE xTrueSum, cSquaredSum, cSum $\leftarrow$ 0 of type TTrue
            \FOR{$1 \leq k \leq n$}
            \STATE $ \text{x}_k \leftarrow$ number generated by rf of type T
            \STATE $\text{xTrue}_k \leftarrow$ $\text{x}_k$ cast to type TTrue
            \STATE $c_k \leftarrow $ as in Table \ref{tab:variableConstruction}
            \STATE cSquaredSum $\leftarrow$ cSquaredSum + $c_k^2$
            \STATE cSum $\leftarrow$ cSum + $c_k$
            \STATE xSum $\leftarrow$ xSum + $\text{x}_k$
            \STATE xTrueSum $\leftarrow$ xTrueSum + $\text{xTrue}_k$
            \ENDFOR
            \STATE compute Equation \ref{eq:cDetBoundGraphs}
            \STATE compute Equation \ref{eq:cProbBoundGraphs}
            \STATE compute relErr
            \RETURN Equations \ref{eq:cDetBoundGraphs}, \ref{eq:cProbBoundGraphs}, and relErr of type TTrue
        \end{algorithmic}
    \end{algorithm}
    \begin{algorithm} %2
        \caption{Computation of \eqref{eq:mProbBoundGraphs} for some $n\in \mathbb{N}$}
        \outputs{} Equations \ref{eq:mProbBoundGraphs} and relErr
        \begin{algorithmic}[1]
            \IF{T is 64 bit Floating Point} 
            \STATE TTrue $\leftarrow$ 256 bit Floating Point
            \ELSE
            \STATE TTrue $\leftarrow$ 64 bit Floating Point
            \ENDIF
            \STATE mSquaredSum, xSum $\leftarrow$ 0 of type T
            \STATE xTrueSum $\leftarrow$ 0 of type TTrue
            \FOR{$1 \leq k \leq n$}
            \STATE $ \text{x}_k \leftarrow$ number generated by rf of type T
            \STATE $\text{xTrue}_k \leftarrow$ $\text{x}_k$ cast to type TTrue
            \IF{$k \neq 1$}
            \STATE $m_{k-1} \leftarrow $ as in Table \ref{tab:variableConstruction}
            \STATE mSquaredSum $\leftarrow$ mSquaredSum + $m_{k-1}^2$
            \ENDIF
            \STATE xSum $\leftarrow$ xSum + $\text{x}_k$
            \STATE xTrueSum $\leftarrow$ xTrueSum + $\text{xTrue}_k$
            \ENDFOR
            \STATE compute Equation \ref{eq:mProbBoundGraphs}
            \STATE compute relErr
            \RETURN Equation \ref{eq:mProbBoundGraphs} and relErr of type TTrue
        \end{algorithmic}
    \end{algorithm}
    
    \pagebreak    
    \subsection{Graphs}
    %use 250pt to get full size, and for side-by-side do subfigures with: \begin{subfigure}[b]{0.4\linewidth}...
    In the following graphs, we start our vector dimensions at the step size since we are assuming exact representation of our generated numbers and this will ensure each data point gives us information. 

    \begin{figure}[H]%1
   	 	\centering
   	 	\includegraphics[width=.75\textwidth]{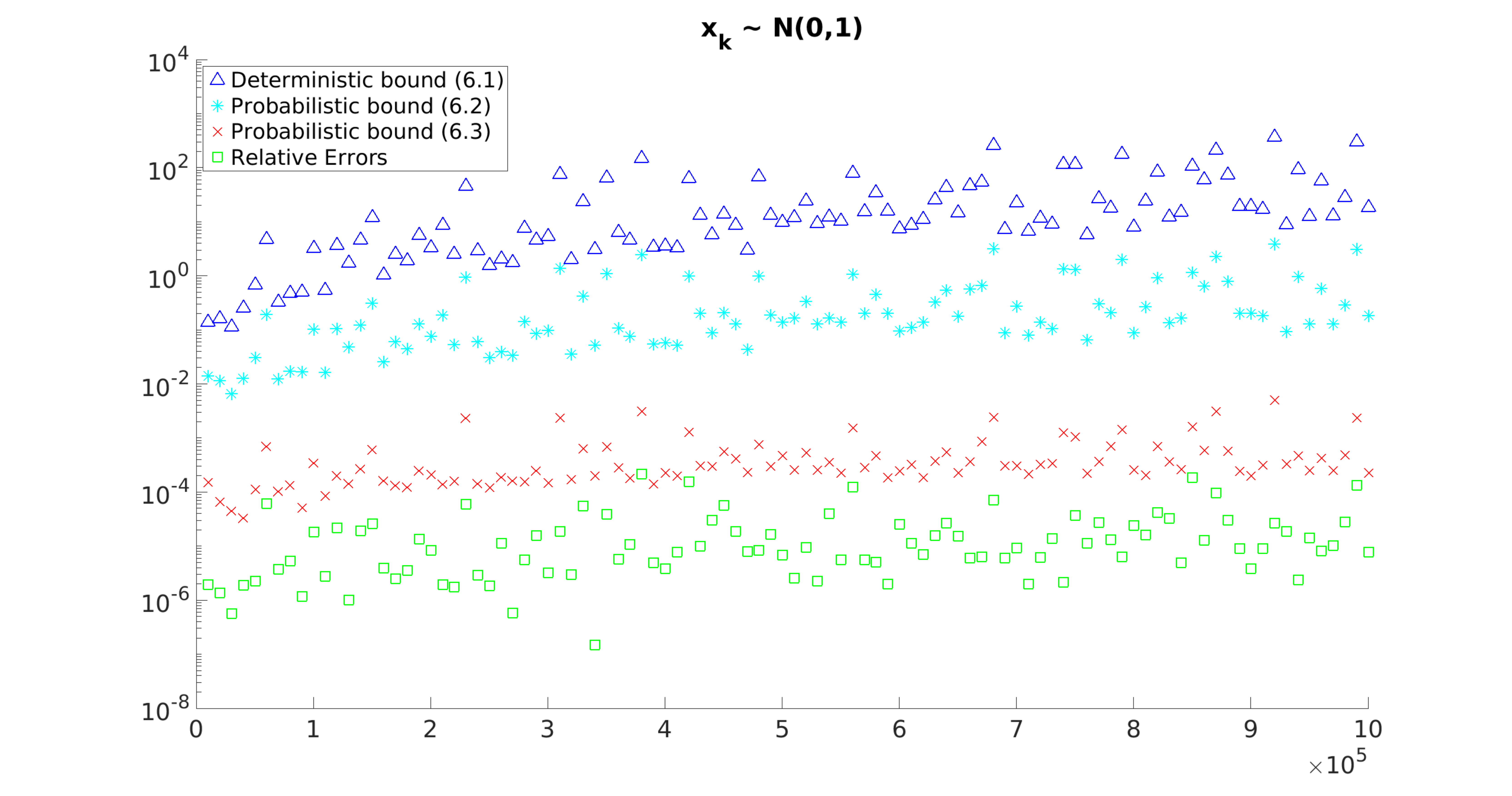}
   	 	\caption{Comparing \eqref{eq:cDetBoundGraphs}, \eqref{eq:cProbBoundGraphs}, and \eqref{eq:mProbBoundGraphs} against the actual relative errors experienced for $10^4 \leq n \leq 10^6$ in steps of $10^4$. Our data is normally distributed random variables in single precision.}\label{fig:32BitNormal}
    \end{figure}
    In Figure \ref{fig:32BitNormal}, our data is restricted to dimensions lower than $10^7$ due to the computational cost of our bounds in large vector dimensions. The best bound, \eqref{eq:mProbBoundGraphs}, is between one and two orders of magnitude of the true relative error, which is by far our best estimator even when compared to \eqref{eq:cProbBoundGraphs}. However, this comes with the trade off of approximately 500 times the computation time on the HPC hardware\footnote{Information from the Systems Administrator is here: \url{https://wp.math.ncsu.edu/it/high-powered-computing-cluster/}{}} listed in Section \ref{sec:Hardware}. Even if higher dimensions are needed, \eqref{eq:cProbBoundGraphs} shows promise since it is about two orders of magnitude tighter than \eqref{eq:cDetBoundGraphs} along with being less computationally expensive.
    \begin{figure}[H]%2
   	 	\centering
   	 	\includegraphics[width=.75\textwidth]{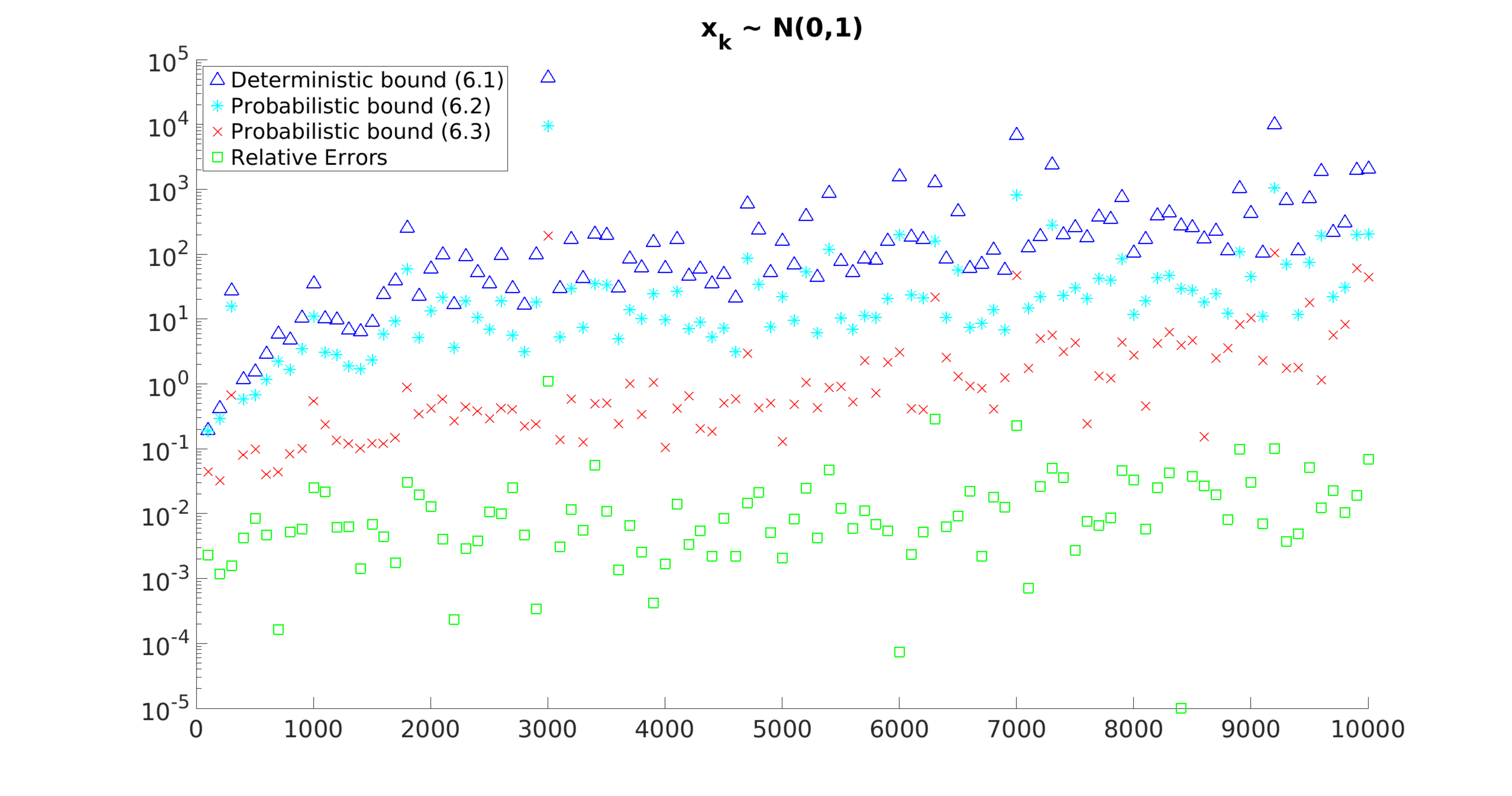}
   	 	\caption{Comparing \eqref{eq:cDetBoundGraphs}, \eqref{eq:cProbBoundGraphs}, and \eqref{eq:mProbBoundGraphs} against the actual relative errors experienced for $10^2 \leq n \leq 10^4$ in steps of $10^2$. Our data is normally distributed random variables in half precision.}\label{fig:16BitNormal}
    \end{figure}
    
    In Figure \ref{fig:16BitNormal}, our dimensions shrink due because  \eqref{eq:cDetBoundGraphs} and \eqref{eq:cProbBoundGraphs} grow larger than one for larger dimensions meaning they no longer predict any accuracy. We notice the same pattern of \eqref{eq:mProbBoundGraphs} as in Figure \ref{fig:32BitNormal}, and the extra tightness becomes exceptionally useful in the half precision case since even at these dimensions, we only estimate the true relative error to be below one with \eqref{eq:mProbBoundGraphs}. Comparing the general trends of \eqref{eq:mProbBoundGraphs} and the true relative error seem to both be similar.
    
    \begin{figure}[H]%3
   	 	\centering
   	 	\includegraphics[width=.75\textwidth]{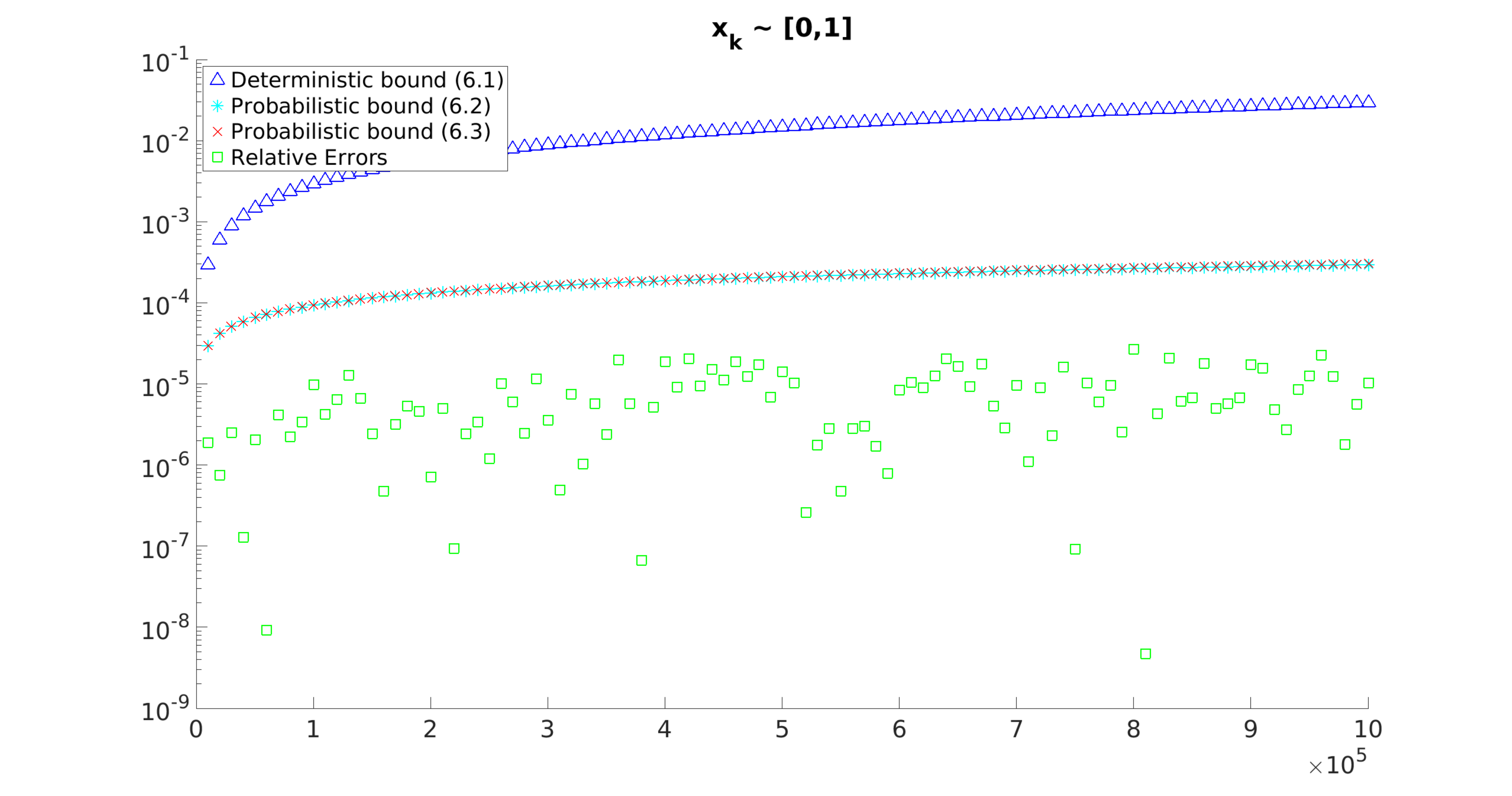}
   	 	\caption{Comparing \eqref{eq:cDetBoundGraphs}, \eqref{eq:cProbBoundGraphs}, and \eqref{eq:mProbBoundGraphs} against the actual relative errors experienced for $10^4 \leq n \leq 10^6$ in steps of $10^4$. Our data is uniformly distributed random variables between 0 and 1 in single precision.}\label{fig:32BitUniform}
    \end{figure}
    
	    In Figure \ref{fig:32BitUniform}, the general behavior of the true relative errors becomes much clearer than in our previous figures. There appears to be a general trend to one, even though we are much below one at our dimensions, so again higher dimensions need to be tested to confirm this claim. The accuracy gains in probabilistic bounds are highlighted by how close \eqref{eq:cProbBoundGraphs} is to \eqref{eq:mProbBoundGraphs}, and the fact that both are two orders of magnitude tighter than \eqref{eq:cDetBoundGraphs} along with being within one order of magnitude of the true relative error.

    \begin{figure}[H]%4
   	 	\centering
   	 	\includegraphics[width=.75\textwidth]{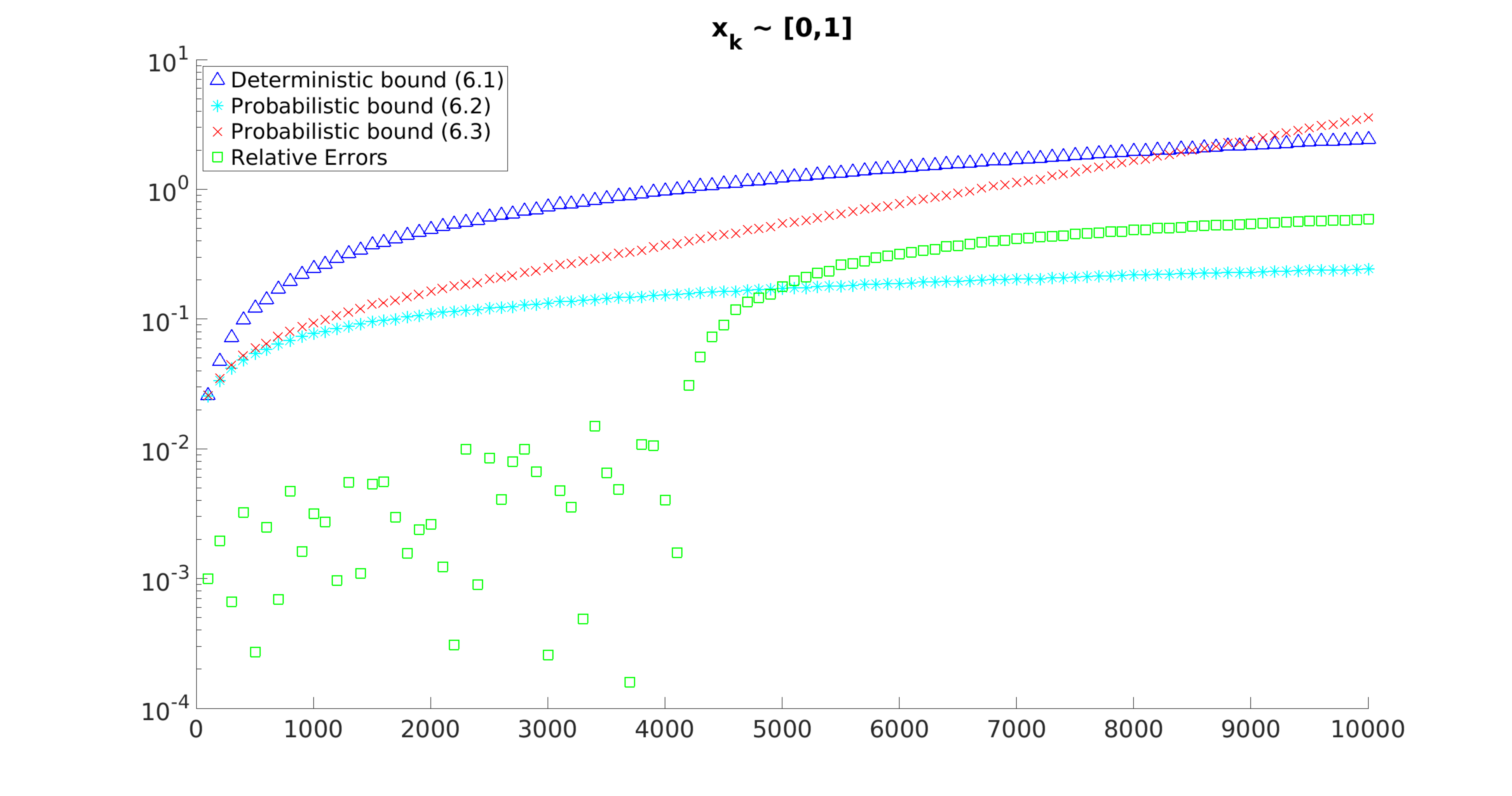}
   	 	\caption{Comparing \eqref{eq:cDetBoundGraphs}, \eqref{eq:cProbBoundGraphs}, and \eqref{eq:mProbBoundGraphs} against the actual relative errors experienced for $10^2 \leq n \leq 10^4$ in steps of $10^2$. Our data is uniformly distributed random variables between 0 and 1 in single precision.}\label{fig:16BitUniform}
    \end{figure}

        In Figure \ref{fig:16BitUniform}, an issue arises. \eqref{eq:cProbBoundGraphs} is no longer an upper bound for the true relative errors. 
        This is either an artifact of the precision being low or a fundamental error in the bound itself, however as previously stated tests for larger 
        dimensions are needed to test this claim. However, \eqref{eq:cDetBoundGraphs} and \eqref{eq:cProbBoundGraphs} are also uninformative because no 
        accuracy is estimated for $n \gtrapprox 6500$.
    
	\section{Conclusion and Future Work}  \label{sec:Conc} We gave roundoff error bounds along with numerical experiments for the sequentially computed sum of $n$ random real numbers. Our experiments confirm that, in most circumstances, the probabilistic bounds are between one and two orders of magnitude more accurate in estimating the actual relative error when compared to known deterministic bounds. We also have found that our bounds come in two types. The first being the most accurate, but takes much more time to compute, \eqref{eq:mProbBoundGraphs}. While the second type is much faster to compute, but not nearly as accurate \eqref{eq:cDetBoundGraphs} and \eqref{eq:cProbBoundGraphs}. A potential fix for this issue would be to implement the sum in such a way that it can be parallelized or to construct a better implementation of \eqref{eq:mProbBoundGraphs}.

        \paragraph{Issues} In Figure \ref{fig:16BitUniform}, \eqref{eq:cProbBoundGraphs} is no longer an upper bound for vectors composed of low precision elements of the same sign, or the failure probability is not the same as our chosen $\delta$ value. A potential solution could be to introduce a bound that depends on the structure of the vector being summed as described in \cite[Section 5.2]{Higham2019}. Further testing for single precision of higher dimensions than $10^7$ are also needed to see if Figure \ref{fig:16BitUniform} is an artifact of the lower precision computations or if it's an issue with the bound itself.

    \bibliographystyle{siam}
    \bibliography{RTGBib}

\begin{thebibliography}{10}

\bibitem{Babuska2018}
{\sc I.~Babu\v{s}ka and G.~S\"{o}derlind}, {\em On roundoff error growth in
  elliptic problems}, ACM Trans. Math. Software, 44 (2018), pp.~Art. 33, 22.

\bibitem{Bareiss1980}
{\sc E.~H. Bareiss and J.~L. Barlow}, {\em Roundoff error distribution in fixed
  point multiplication}, BIT, 20 (1980), pp.~247--250.

\bibitem{Barlow1985}
{\sc J.~L. Barlow and E.~H. Bareiss}, {\em On roundoff error distributions in
  floating point and logarithmic arithmetic}, Computing, 34 (1985),
  pp.~325--347.

\bibitem{Barlow1985a}
\leavevmode\vrule height 2pt depth -1.6pt width 23pt, {\em Probabilistic error
  analysis of {G}aussian elimination in floating point and logarithmic
  arithmetic}, Computing, 34 (1985), pp.~349--364.

\bibitem{Chung2006}
{\sc F.~Chung and L.~Lu}, {\em Concentration inequalities and {Martingale}
  inequalities: A survey}, Internet Math., 3 (2006), pp.~79--127.

\bibitem{Higham2002}
{\sc N.~J. Higham}, {\em Accuracy and stability of numerical algorithms},
  Society for Industrial and Applied Mathematics (SIAM), Philadelphia, PA,
  second~ed., 2002.

\bibitem{Higham2019}
{\sc N.~J. Higham and T.~Mary}, {\em A new approach to probabilistic rounding
  error analysis}, SIAM J. Sci. Comput., 41 (2019), pp.~A2815--A2835.

\bibitem{Higham2020}
{\sc N.~J. Higham and T.~Mary}, {\em {Sharper Probabilistic Backward Error
  Analysis for Basic Linear Algebra Kernels with Random Data}}.
\newblock working paper or preprint, Jan. 2020.

\bibitem{Hull1966}
{\sc T.~E. Hull and J.~R. Swenson}, {\em Tests of probabilistic models for the
  propagation of roundoff errors}, Communications of the ACM, 9 (1966),
  pp.~108--113.

\bibitem{IEEE}
{\sc {IEEE Computer Society}}, {\em {IEEE} standard for floating-point
  arithmetic, {IEEE} {Std} 754–2019"}, July 2019.

\bibitem{Ipsen2020}
{\sc I.~C.~F. Ipsen and H.~Zhou}, {\em Probabilistic error analysis for inner
  products}, SIAM J. Matrix Anal. Appl., to appear (2020).

\bibitem{Kahan1996}
{\sc W.~Kahan}, {\em The improbability of probabilistic error analyses for
  numerical computations},  (1996).

\bibitem{Mitzenmacher2006}
{\sc M.~Mitzenmacher and E.~Upfal}, {\em Probability and Computing: Randomized
  Algorithms and Probabilistic Analysis}, Cambridge University Press, New York,
  2006.

\bibitem{Tienari1970}
{\sc M.~Tienari}, {\em A statistical model of roundoff error for varying length
  floating-point arithmetic}, Nordisk Tidskr. Informationsbehandling (BIT), 10
  (1970), pp.~355--365.

\bibitem{Neumann1947}
{\sc J.~{von Neumann} and H.~H. Goldstine}, {\em Numerical inverting of
  matrices of high order}, Bull. Amer. Math. Soc., 53 (1947), pp.~1021--1099.

\bibitem{FloatExample}
{\sc Wikimedia}, {\em File:float example.svg}.
\newblock Wikimedia Upload, Jan. 2008.
\newblock \url{https://commons.wikimedia.org/wiki/File:Float_example.svg}.
  License: \url{https://creativecommons.org/licenses/by-sa/3.0/}.

\end{thebibliography}
\end{document}